\documentclass[12pt]{amsart}

\usepackage{amsmath,amsfonts,amssymb,mathabx,latexsym,breqn,stmaryrd,mathtools,mathrsfs}
\usepackage{enumitem}
\usepackage{subcaption}
\usepackage[usenames, dvipsnames,svgnames]{xcolor}
\usepackage[backref=page]{hyperref}
\usepackage{ytableau}
\usepackage{array}
\usepackage{tikz}
\usetikzlibrary{calc, shapes, backgrounds,arrows.meta,positioning,plotmarks,decorations.markings}
\tikzset{>=Straight Barb,
  head/.style = {fill = white, text=black},
  plaque/.style = {draw, rectangle, minimum size = 10mm}, 
  pil/.style={->,thick},
  pilpil/.style={<->,thick},
  junct/.style = {draw,circle,inner sep=0.5pt,outer sep=0pt, fill=black}
  }
 
\newcommand{\dynkinradius}{.15cm}
\newcommand{\dynkinstep}{.58cm}
\newcommand{\dynkinnormal}[2]{\fill (\dynkinstep*#1,\dynkinstep*#2) circle (\dynkinradius);}
\newcommand{\dynkinmin}[2]{\filldraw[fill=LightSkyBlue,draw=black] (\dynkinstep*#1,\dynkinstep*#2) circle (\dynkinradius);}
\newcommand{\dynkinline}[4]{\draw[thin] (\dynkinstep*#1,\dynkinstep*#2) -- (\dynkinstep*#3,\dynkinstep*#4);}
\newcommand{\dynkindots}[4]{\draw[dotted,very thick] (\dynkinstep*#1,\dynkinstep*#2) -- (\dynkinstep*#3,\dynkinstep*#4);}
\newcommand{\dynkindoubleline}[4]{\draw[double,double distance between line centers=0.19em,postaction={decorate}] (\dynkinstep*#1,\dynkinstep*#2) -- (\dynkinstep*#3,\dynkinstep*#4);}

\newenvironment{dynkin}{\begin{tikzpicture}[decoration={markings,mark=at position 0.7 with {\arrow{>[scale=.7]}}}]}
{\end{tikzpicture}}

\usepackage{enumitem}
\newlist{arrowlist}{itemize}{1}
\setlist[arrowlist]{label=$\Rightarrow$}
\newcommand{\x}{\ensuremath{\mathsf{x}}}
\newcommand{\y}{\ensuremath{\mathsf{y}}}
\newcommand{\z}{\ensuremath{\mathsf{z}}}

\newtheorem{theorem}{Theorem}[section]
\newtheorem{lemma}[theorem]{Lemma}
\newtheorem{proposition}[theorem]{Proposition}
\newtheorem{corollary}[theorem]{Corollary}
\newtheorem{conjecture}[theorem]{Conjecture}

\theoremstyle{definition}
\newtheorem{definition}[theorem]{Definition}
\newtheorem{examplex}[theorem]{Example}
\newenvironment{example}
  {\pushQED{\qed}\examplex}
  {\popQED\endexamplex}

\theoremstyle{remark}
\newtheorem{remark}[theorem]{Remark}

\numberwithin{equation}{section}

\newcommand{\inc}{\ensuremath{\mathrm{Inc}}}
\newcommand{\incgl}{\inc_{\mathrm{gl}}}
\newcommand{\pro}{\mathfrak{pro}}

\newcommand{\rank}{\ensuremath{\mathrm{rk}}}

\newcommand{\pp}{\ensuremath{\mathsf{PP}}}
\newcommand{\deflate}{\ensuremath{\mathsf{Defl}}}
\newcommand{\inflate}{\ensuremath{\mathsf{VecInfl}}}

\newcommand{\tinflate}{\ensuremath{\mathsf{Infl}}}
\newcommand{\content}{\ensuremath{\mathsf{Con}}}
\newcommand{\compress}{\ensuremath{\mathsf{DeflCon}}}
\newcommand{\kbk}{\rho}

\newcommand{\uu}{\mathcal{I}}

\begin{document}


\title[Orbits of plane partitions]{Orbits of plane partitions of exceptional Lie type}  

\author[H. Mandel]{Holly Mandel}
\address[HM]{Department of Mathematics, University of California, Berkeley, \linebreak Berkeley, CA 94720}
\email{holly.mandel@berkeley.edu}

\author[O. Pechenik]{Oliver Pechenik}
\address[OP]{Department of Mathematics, University of Michigan, Ann Arbor, MI 48109}
\email{pechenik@umich.edu}

\subjclass[2010]{Primary 05A15; Secondary 05E18, 06A07, 17B25}

\date{\today}


\keywords{plane partition, increasing tableau, promotion, rowmotion, exceptional Lie algebra, minuscule poset, cyclic sieving phenomenon}

\begin{abstract}
For each minuscule flag variety $X$, there is a corresponding minuscule poset, describing its Schubert decomposition. We study an action on plane partitions over such posets, introduced by P.~Cameron and D.~Fon-der-Flaass (1995).
For plane partitions of height at most $2$, D.~Rush and X.~Shi (2013) proved an instance of the cyclic sieving phenomenon, completely describing the orbit structure of this action. They noted their result does not extend to greater heights in general; however, when $X$ is one of the two minuscule flag varieties of exceptional Lie type $E$, they conjectured explicit instances of cyclic sieving for all heights.

We prove their conjecture in the case that $X$ is the Cayley-Moufang plane of type $E_6$. For the other exceptional minuscule flag variety, the Freudenthal variety of type $E_7$, we establish their conjecture for heights at most $4$, but show that it fails generally. We further give a new proof of an unpublished cyclic sieving of D.~Rush and X.~Shi (2011) for plane partitions of any height in the case $X$ is an even-dimensional quadric hypersurface. 
Our argument uses ideas of K.~Dilks, O.~Pechenik, and J.~Striker (2017) to relate the action on plane partitions to combinatorics derived from $K$-theoretic Schubert calculus. 
\end{abstract}

\maketitle

%
\section{Introduction}
%
\label{sec:introduction}

The \emph{minuscule posets} are a remarkable collection of partially-ordered sets that arise naturally from the representation theory of Lie algebras or alternatively from the Schubert calculus of generalized flag varieties. We study dynamical enumerative properties of plane partitions over these posets.

A special case is the set of ordinary plane partitions that fit inside a fixed rectangular box. In this context, P.~Cameron and D.~Fon-der-Flaass \cite{Cameron.Fonderflaass} initiated the study of a combinatorially-natural operator $\Psi$. This operator is now generally known as \emph{rowmotion} and has become a subject of intense study (cf., e.g., \cite{Panyushev,Striker.Williams,Armstrong.Stump.Thomas,Rush.Shi,Einstein.Propp,Propp.Roby,Grinberg.Roby:2,Grinberg.Roby:1,DPS,Vorland, Dilks.Striker.Vorland}).
We will describe minuscule posets and the operation of rowmotion in Sections~\ref{sec:minuscule} and \ref{sec:rowmotion}, respectively.

For any poset $P$, let $\pp^k(P)$ denote the set of plane partitions of height at most $k$ over $P$, or equivalently, the set of order ideals in the product $P \times {\bf k}$ of $P$ with a chain poset of $k$ elements. Let $f_P^k$ denote the generating function that enumerates the elements of $\pp^k(P)$  by cardinality, so 
$f_P^k(q) \coloneqq \sum_{\mathcal{I} \in \pp^k(P)} q^{|\mathcal{I}|}.$ In the special case $k \leq 2$ and $P$ minuscule, D.~Rush and X.~Shi \cite{Rush.Shi} showed that $f_P^k$ also encodes the orbit structure of rowmotion via an instance of the {\bf cyclic sieving phenomenon} (introduced by V.~Reiner, D.~Stanton, and D.~White \cite{Reiner.Stanton.White}). Thus for $k \leq 2$, the number of minuscule plane partitions fixed by the $d$-fold application of rowmotion $\Psi^{\circ d}$ is the evaluation of the polynomial $f_P^k$ at $\zeta^d$, where $\zeta$ is any primitive $n$th root of unity and $n$ is the period of $\Psi$ on $\pp^k(P)$. (Since some of our other operators have superscripts in their names, we denote the $d$-fold composition of an operator $\tau$ by $\tau^{\circ d}$.)

It was noted in \cite{Rush.Shi} that this instance of cyclic sieving does not extend to the case $k\geq 3$ for general minuscule posets.
However, D.~Rush and X.~Shi conjectured the following. (The posets in question are illustrated in Figure~\ref{fig:min_poset_E}.)
\begin{conjecture}[{\cite[Conjecture~11.1]{Rush.Shi}}]\label{conj:rush.shi}
Let $P$ be one of the two minuscule posets associated to an exceptional Lie algebra of type $E$ and let $k \in \mathbb{Z}_{\geq 0}$. Then $f_P^k$ is a cyclic sieving polynomial for the action of $\Psi$ on $\pp^k(P)$.
\end{conjecture}

\begin{figure}[h]
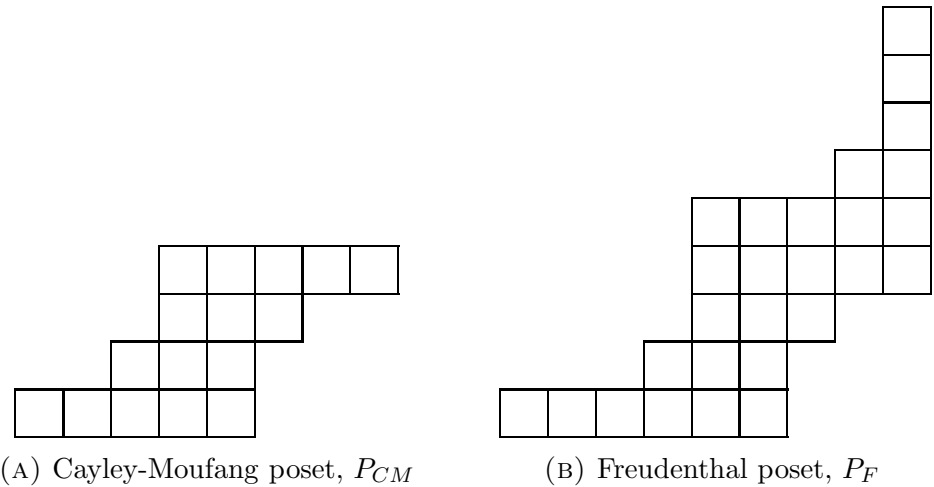

	\begin{subfigure}[b]{0.37\textwidth}
		\centering
		\ydiagram{3+5,3+3,2+3,5}
		\caption{Cayley-Moufang poset, $P_{CM}$}
	\end{subfigure} \hspace{8mm} 
	\begin{subfigure}[b]{0.37\textwidth}
		\centering
		\ydiagram{8+1,8+1,8+1,7+2,4+5,4+5,4+3,3+3,6}
		\caption{Freudenthal poset, $P_F$}
	\end{subfigure}
\caption{The two minuscule posets associated to exceptional Lie algebras of type $E$. Here, we have drawn the posets to resemble Young diagrams in Cartesian (``French'') orientation; the boxes are the elements of the poset and each box covers the box immediately below it and the box immediately to its left (if such boxes exist). Hence the minimal element of each poset is the box at the far left of the bottom row.
 The Cayley-Moufang poset $P_{CM}$ is associated to $E_6$, while the Freudenthal poset $P_F$ is associated to $E_7$. (There is no minuscule poset associated to $E_8$.)}
\label{fig:min_poset_E}
\end{figure}

Our main result is to completely resolve Conjecture~\ref{conj:rush.shi}. 
\begin{theorem}\label{thm:exceptionals}
Conjecture~\ref{conj:rush.shi} holds for the $E_6$ minuscule poset $P_{CM}$ (cf.~Figure~\ref{fig:min_poset_E}A) and all $k$, but holds for the $E_7$ minuscule poset $P_F$ (cf.~Figure~\ref{fig:min_poset_E}B) only when $k \leq 4$. 
\end{theorem}  
Verification of Conjecture~\ref{conj:rush.shi} in the the cases $k \leq 4$ for $P_{CM}$ and $k \leq 3$ for $P_F$ was previously reported in \cite{Rush.Shi}. Our new results are therefore:
\begin{itemize}
\item cyclic sieving for $P_{CM}$ when $k > 4$,
\item cyclic sieving for $P_F$ when $k = 4$, and
\item failure of the conjectured cyclic sieving for $P_F$ when $k > 4$.
\end{itemize}

Our approach to proving Theorem~\ref{thm:exceptionals} is to use the ideas of K.~Dilks, O.~Pechenik, J.~Striker and C.~Vorland \cite{DPS, Dilks.Striker.Vorland} to relate the action of $\Psi$ to the action of \emph{$K$-promotion} on \emph{increasing tableaux}. $K$-promotion was first studied in \cite{Pechenik}, building on combinatorial tools for $K$-theoretic Schubert calculus due to H.~Thomas and A.~Yong \cite{Thomas.Yong:K}. We show that the action of $K$-promotion is controlled by its behavior on a finite subset of increasing tableaux. By understanding the orbit structure of this subset, we are able to determine the complete orbit structure, thereby establishing Theorem~\ref{thm:exceptionals}.

Having developed these methods, it becomes straightforward to prove the following additional result.
\begin{theorem}\label{thm:propeller}
Let $P$ be a minuscule poset associated to an even-dimensional quadric of type $D_{p+1}$ (cf.~Figure~\ref{fig:min_poset}C) and let $k \in \mathbb{Z}_{\geq 0}$. Then $f_P^k$ is a cyclic sieving polynomial for the action of $\Psi$ on $\pp^k(P)$.
\end{theorem}
Theorem~\ref{thm:propeller} was previously announced by D.~Rush and X.~Shi \cite[Theorem~10.1]{Rush.Shi}; however, they omitted their proof \cite[\textsection 10]{Rush.Shi:report} from the published paper. We believe that our alternative proof of Theorem~\ref{thm:propeller} via $K$-theoretic combinatorics provides different insight.

\begin{remark}
Often instances of cyclic sieving can be proven using representation-theoretic techniques \cite{Reiner.Stanton.White, Rhoades:thesis}, and when cyclic sieving is established in a more direct fashion, as we do here, it may be a clue toward new underlying algebra (cf.\ \cite{Rhoades:skein}). The results in Theorems~\ref{thm:exceptionals} and \ref{thm:propeller} perhaps suggest the existence of new symmetric group module structures on the sets $\pp^k(P)$. Developing such representations would be an interesting direction for future work. In particular, we do not have a good understanding of why the posets $P_{CM}$ and $P_F$ behave so differently in Theorem~\ref{thm:exceptionals}, even though the associated algebra and geometry seems very similar; a representation-theoretic construction might shed light on this mystery. It is also possible that the difference between $P_{CM}$ and $P_F$ in Theorem~\ref{thm:exceptionals} could be explained via monodromy in real Schubert calculus by vastly extending the geometric constructions of \cite{Levinson}; from conversations with J.~Levinson and K.~Purbhoo, it seems that there are many obstacles, however, to developing the necessary geometry to give such a geometric explanation of the results here.
\end{remark}

This paper is organized as follows. In Section~\ref{sec:minuscule}, we define the minuscule posets, recalling their classification and other properties we will use. Section~\ref{sec:rowmotion} gives the precise definition of rowmotion and, following \cite{DPS,Dilks.Striker.Vorland}, notes the close relation between rowmotion and $K$-promotion. We then develop new tools for understanding the orbit structure of $K$-promotion in Sections~\ref{sec:inflation}, \ref{sec:interaction}, and \ref{sec:period}. Specifically, Section~\ref{sec:inflation} introduces the operations of \emph{inflation} and \emph{deflation}, retracting the set of increasing tableaux onto the finite subset of \emph{gapless tableaux}. 
In Section~\ref{sec:interaction}, we recall the precise definition of $K$-promotion on increasing tableaux and show how $K$-promotion is governed by its restriction to gapless tableaux. Section~\ref{sec:period} uses this information to determine the period of $K$-promotion on general increasing tableaux.
Finally, Section~\ref{sec:arithmetic} combines these ideas to prove Theorems~\ref{thm:exceptionals} and \ref{thm:propeller}.

\section{Minuscule posets}\label{sec:minuscule}

Let ${\sf G}$ be a complex connected reductive Lie group with maximal torus ${\sf T}$. Denote by $W$ the Weyl group $N_{\sf G}({\sf T})/{\sf T}$. The root system $\Phi$ of ${\sf G}$ may be partitioned $\Phi^+ \sqcup \Phi^-$ into positive and negative roots according to a choice $\Delta$ of simple roots. There is a natural poset structure on $\Phi^+$ obtained as the transitive closure of the covering relation $\alpha \lessdot \beta$ if and only if $\beta - \alpha \in \Delta$. The choice of bipartition of $\Phi$ into positive and negative roots further specifies a choice of a Borel subgroup ${\sf B}_+ \subset {\sf G}$ and an opposite Borel subgroup ${\sf B}_- \subset {\sf G}$ with ${\sf B}_+ \cap {\sf B}_- = {\sf T}$.

We say $\delta \in \Delta$ is {\bf minuscule} if for every $\alpha \in \Phi^+$, $\delta^\vee$ appears with multiplicity at most $1$ in the simple coroot expansion of $\alpha^\vee$. The classification of minuscule roots is well known and is illustrated in Figure~\ref{fig:minuscule} in terms of Dynkin diagrams.

\begin{figure}[h]
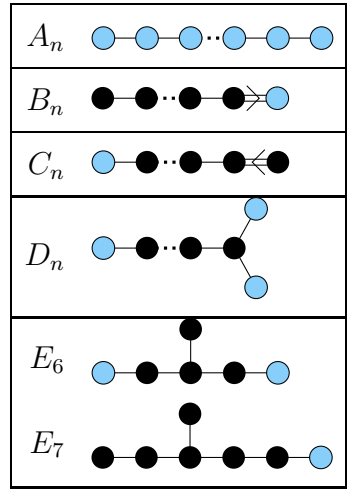

 \renewcommand*{\arraystretch}{1.6}
\begin{tabular}{|>{$}r<{$}m{3.2cm}|}
\hline
A_n &
  \begin{dynkin}
    \dynkinline{1}{0}{3}{0};
    \dynkindots{3}{0}{4}{0};
    \dynkinline{4}{0}{6}{0};
    \foreach \x in {1,...,6}
    {\dynkinmin{\x}{0}}
  \end{dynkin}
 \\  \hline B_n &
  \begin{dynkin}
    \dynkinline{1}{0}{2}{0};
    \dynkindots{2}{0}{3}{0};
    \dynkinline{3}{0}{4}{0};
    \dynkindoubleline{4}{0}{5}{0};
    \dynkinmin{5}{0};
    \foreach \x in {1,...,4}
    {
        \dynkinnormal{\x}{0}
    }
  \end{dynkin}
\\ \hline   C_n 
&
  \begin{dynkin}
    \dynkinline{1}{0}{2}{0};
    \dynkindots{2}{0}{3}{0};
    \dynkinline{3}{0}{4}{0};
    \dynkindoubleline{5}{0}{4}{0};
    \dynkinmin{1}{0};
    \foreach \x in {2,...,5}
    {
        \dynkinnormal{\x}{0}
    }
  \end{dynkin}
\\ \hline 
D_n
&
  \begin{dynkin}
    \foreach \x in {2,...,4}
    {
        \dynkinnormal{\x}{0}
    }
        \dynkinline{3}{0}{4}{0}
    \dynkinline{4}{0}{4.5}{.9}
    \dynkinline{4}{0}{4.5}{-.9}
        \dynkinline{1}{0}{2}{0}
    \dynkinmin{4.5}{.9}
    \dynkinmin{4.5}{-.9}
        \dynkinmin{1}{0}
    \dynkindots{2}{0}{3}{0}
  \end{dynkin} 
\\  \hline  E_6 
&
  \begin{dynkin}
    \foreach \x in {2,...,4}
    {
        \dynkinnormal{\x}{0}
    }
        \dynkinline{1}{0}{5}{0}
    \dynkinline{3}{0}{3}{1}
    \dynkinmin{1}{0}
    \dynkinmin{5}{0}
    \dynkinnormal{3}{1}
  \end{dynkin}
\\    E_7
&
  \begin{dynkin}
    \foreach \x in {1,...,5}
    {
        \dynkinnormal{\x}{0}
    }
        \dynkinline{1}{0}{6}{0}
    \dynkinline{3}{0}{3}{1}
    \dynkinmin{6}{0}
    \dynkinnormal{3}{1}
  \end{dynkin} \\
  \hline
\end{tabular}
 \caption{In each of the finite-type Dynkin diagrams above, each minuscule root is marked as a pale blue disk, while the non-minuscule simple roots are marked in black. In type $A_n$, every node is minuscule, while in the other types only the indicated leaves are minuscule. The remaining finite-type Dynkin diagrams are omitted because they have no minuscule nodes.}\label{fig:minuscule}
\end{figure}

For each minuscule simple root $\delta$, there is an associated {\bf minuscule poset} $P_\delta$ obtained as the subposet of $\Phi^+$ induced on those positive roots $\alpha$ where $\delta$ appears with nonzero coefficient in the simple root expansion of $\alpha$. 

Alternatively, one may obtain the minuscule posets via the geometry of certain generalized flag varieties. If ${\sf P}_\delta \supset {\sf B_+}$ denotes the maximal parabolic subgroup of ${\sf G}$ associated to the minuscule simple root $\delta$, then the space $X = {\sf G} / {\sf P}_\delta$ is called a {\bf minuscule variety}. The minuscule varieties are smooth projective varieties with many additional nice geometric properties (cf.,~e.g.,~\cite{Billey.Lakshmibai} for details). The natural action of the Borel subgroup ${\sf B}_+$ on $X$ has finitely many orbits, whose Zariski closures are the {\bf Schubert varieties}. Given two Schubert varieties in $X$, it is known that they are either disjoint or else one is a subset of the other. Indeed the poset $Y_X$ of Schubert varieties of $X$ with respect to inclusion is a distributive lattice and its corresponding poset of join irreducibles is isomorphic to the minuscule poset $P_\delta$ for the minuscule simple root $\delta$, as constructed above.

The minuscule posets are completely classified. We illustrate them here in Figures~\ref{fig:min_poset_E} and \ref{fig:min_poset}. Our focus will be on the {\bf propellers}, the {\bf Cayley-Moufang poset}, and the {\bf Freudenthal poset} as shown in Figures~\ref{fig:min_poset}C, \ref{fig:min_poset_E}A, and \ref{fig:min_poset_E}B, respectively. The propeller $P_p$ with $2p$ elements ($p \geq 3$) is associated to the minuscule node that is not adjacent to the trivalent node in the $D_{p+1}$ Dynkin diagram, the Cayley-Moufang poset $P_{CM}$ is associated to either of the two minuscule simple roots for $E_6$, and the Freudenthal poset $P_F$ is associated to the unique minuscule simple root for $E_7$ (cf.\ Figure~\ref{fig:minuscule}). The corresponding minuscule varieties are, respectively, even-dimensional quadric hypersurfaces, the octonionic projective plane (or Cayley-Moufang plane), and the Freudenthal variety. The remaining minuscule roots yield minuscule posets that are \emph{rectangles} or \emph{shifted staircases} (illustrated in Figure~\ref{fig:min_poset}A and B). These correspond respectively to type $A$ Grassmannians and to maximal orthogonal Grassmannians; we will not consider these posets further in this paper (except as convenient examples), since the orders of rowmotion and $K$-promotion are generally unknown for them.

A poset that is linearly ordered is called a {\bf chain}; we denote the chain on $k$ elements by $\mathbf{k}$. By a {\bf plane partition} of height at most $k$ over a poset $P$, we mean an order ideal of the product poset $P \times \mathbf{k}$.  Clearly, such a plane partition may be identified with a weakly order-reversing map $\pi : P \to \mathbf{k}$ (i.e., a map $\pi$ such that $\x \leq \x'$ implies $\pi(\x) \geq \pi(\x')$). 

\begin{remark}\label{rem:gaussian}
The poset $P$ is called {\bf Gaussian} if the generating function $f_P^k(q)$ may be expressed in the form
\[
f_P^k(q) = \frac{(1 - q^{h_1 + k})(1 - q^{h_2 + k})\cdots(1 - q^{h_t + k})}{(1 - q^{h_1})(1 - q^{h_2})\cdots(1 - q^{h_t})},
\]
for some nonnegative integers $t, h_1, h_2, \dots, h_t \in \mathbb{Z}_{\geq 0}$ independent of $k$.
R.~Proctor showed that every minuscule poset is Gaussian \cite{Proctor}. Indeed, it is conjectured that there are no other connected Gaussian posets. In the case of a minuscule poset $P$, one can in fact take $t = |P|$ and, for $\x \in P$, take $h_\x = \rank(\x) + 1$, where $\rank(\x)$ denotes the length of the largest chain in $P$ with maximum element $\x$. (The {\bf length} of a chain is the number of covering relations; that is, $x_0 < x_1 < \dots < x_n$ is a chain of length $n$.) Given the beautiful form of these generating functions, it is perhaps not surprising that $f_P^k$ should play a role in some instances of cyclic sieving, as in Theorems~\ref{thm:exceptionals} and \ref{thm:propeller}. 
\end{remark}

\begin{figure}[h]
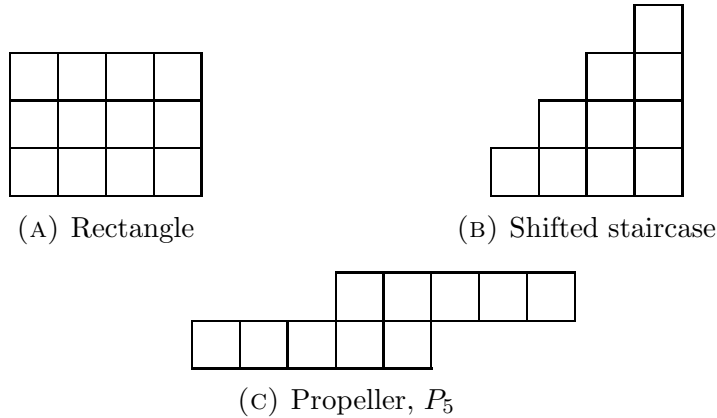

	\begin{subfigure}[b]{0.27\textwidth}
		\centering
		\ydiagram{4,4,4}
		\caption{Rectangle}
	\end{subfigure}
	\hspace{2cm}
	\begin{subfigure}[b]{0.27\textwidth}
		\centering
		\ydiagram{3+1,2+2,1+3,4}
		\caption{Shifted staircase}
	\end{subfigure} \\
	\vspace{3mm}
	\begin{subfigure}[b]{0.27\textwidth}
		\centering
		\ydiagram{3+5,5}
		\caption{Propeller, $P_5$}
	\end{subfigure}
\caption{Together with the two posets shown in Figure~\ref{fig:min_poset_E}, these are exemplars of all five families of minuscule posets, shown in Cartesian orientation. The elements of each poset are the boxes, and each box is covered by any box immediately above it or immediately to its right. Rectangles may have arbitrary height and width. Shifted staircases have arbitrary width, and height equal to their width. The propeller $P_p$ consists of two rows of length $p$, overlapping by two boxes in the center. The Cayley-Moufang and Freudenthal posets of Figure~\ref{fig:min_poset_E} are exceptional, forming singleton families.}\label{fig:min_poset}
\end{figure}

\section{The rowmotion operator $\Psi$}\label{sec:rowmotion}

If $P$ is any finite poset, let $J(P)$ denote the set of its order ideals. For $\uu \in J(P)$, we define $\Psi(\uu) \in J(P)$ to be the order ideal generated by the minimal elements of the complement $P - \uu$. This operator $\Psi$ is closely related to those described in \cite{Brouwer.Schrijver,Duchet,Cameron.Fonderflaass} in different contexts. It was first considered explicitly as an action on order ideals by J.~Striker and N.~Williams \cite{Striker.Williams}. We follow them in referring to $\Psi$ as {\bf rowmotion}.

We now observe that results of \cite{DPS} and \cite{Dilks.Striker.Vorland} enable us to study the rowmotion action $\Psi$ via the action of $K$-promotion on increasing tableaux.

\begin{definition}
Let $\lambda$ be an order ideal of a minuscule poset, considered as an generalized Young diagram in Cartesian orientation as in Figures~\ref{fig:min_poset_E} and \ref{fig:min_poset}. An {\bf increasing tableau} of shape $\lambda$ is an assignment of a positive integer to each box of $\lambda$, such that entries strictly increase from left to right along rows and strictly increase from bottom to top going up columns. Let $\inc^m(\lambda)$ denote the set of increasing tableaux of shape $\lambda$ with all entries at most $m$. We identify $\inc^m(\lambda)$ with the set of strictly order-preserving maps from the poset $\lambda$ to the chain poset $\mathbf{m} = \{1, 2, \dots, m\}$ on $m$ elements. For an example of an increasing tableau, see Figure~\ref{fig:inctab}.
\end{definition}

\begin{figure}[h]
\ytableaushort{\none \none \none \none \none \none \none \none {23}, \none \none \none \none \none \none \none \none {22}, \none \none \none \none \none \none \none \none {21}, \none \none \none \none \none \none \none {19} {20}, \none  \none \none \none {11} {12}{15}{18}{19}, \none \none \none \none 9 {11}{13}{16}{17}, \none \none \none \none 8{10}{11}, \none \none \none 679,123456}
\caption{A representative increasing tableau $T \in \inc^{23}(P_F)$.}\label{fig:inctab}
\end{figure}

From the combinatorics of $K$-theoretic Schubert calculus one obtains a $K$-promotion operator $\pro^m$ on $\inc^m(\lambda)$ that directly extends M.-P.~Sch\"utzenberger's classical definition of promotion \cite{Schutzenberger:promotion}. We will define $\pro^m$ in Section~\ref{sec:interaction}. For a finite poset $P$, let $\rank(P)$ denote the length of the longest chain in $P$. (This is one less than the definition of $\rank$ in \cite{Dilks.Striker.Vorland}.) 

\begin{proposition}
Let $P \in \{ P_{CM}, P_F, P_p \}$ and $k \in \mathbb{Z}_{\geq 0}$.
There is an equivariant bijection between $\inc^{k+\rank(P)+1}(P)$ under $K$-promotion and $J(P \times {\bf k})$ under $\Psi$.
\end{proposition}
\begin{proof}
By \cite[Corollary~5.2]{Dilks.Striker.Vorland}, there is an equivariant bijection between the sets $\inc^{k+\rank(P)+1}(P)$ under $K$-promotion and $J(\Gamma_1(P,k+\rank(P)+1))$  under $\Psi$, where $\Gamma_1(P,k+\rank(P)+1)$ is an auxilliary poset constructed from $P$ in \cite[\textsection 3.1]{Dilks.Striker.Vorland}.
For any $\x \in P$, note that the maximal length of a chain through $\x$ is independent of $\x$. (Specifically, this maximal length is $\rank(P)$; we have $\rank(P_{CM}) = 10$, $\rank(P_F)= 16$, and $\rank(P_p) = 2p-2$.)
Hence \cite[Corollary~3.28]{Dilks.Striker.Vorland} applies, and in these cases we have the isomorphism of posets $\Gamma_1(P,k+\rank(P)+1) \cong P \times {\bf k}$.
\end{proof}

\begin{corollary}\label{cor:multisets}
Let $P \in \{ P_{CM}, P_F, P_p \}$ and $k \in \mathbb{Z}_{\geq 0}$.
Then the multiset of cardinalities of $\Psi$-orbits on $J(P \times \mathbf{k})$ and the multiset of cardinalities of $\pro^{\rank(P)+ k+1}$-orbits on $\inc^{\rank(P)+ k+1}(\lambda)$ are equal. \qed
\end{corollary}

\section{Inflation and deflation}\label{sec:inflation}
In light of Corollary~\ref{cor:multisets}, we now turn to a more thorough study of increasing tableaux. Let $\lambda$ be an order ideal in a minuscule poset.
Suppose $T \in \inc^m(\lambda)$ and consider $T$ as a strictly order-preserving map from $\lambda$ to the chain poset ${\bf m}$. We say that $T$ is {\bf gapless} if this map is surjective and {\bf gappy} otherwise. We write $\incgl^m(\lambda)$ for the subset of all gapless tableaux in $\inc^m(\lambda)$. Notice that the set 
\[
\incgl(\lambda) \coloneqq \coprod_{m} \incgl^m(\lambda)
\]
is finite. This fact will be critical to our proofs of Theorems~\ref{thm:exceptionals} and~\ref{thm:propeller}.

For $T \in \inc^m(\lambda)$, let $m_T$ be the number of distinct labels in $T$. For each $m$, we define the {\bf deflation} map \[\deflate^m : \inc^m(\lambda) \to \coprod_{0 \leq n \leq m} \incgl^n(\lambda)\] by
\[
[\deflate^m(T)](\x) =
\# \{ h \in \mathrm{range}(T): h \leq T(\x) \} ,
\]
for $T \in \inc^m(\lambda)$ and $\x \in \lambda$. Note that $\deflate^m(T) \in \incgl^{m_T}(\lambda)$ and that $m_T \leq m$.

For nonnegative integers $j$ and $k$, let $\binom{[j]}{k}$ denote the collection of binary vectors of length $j$ with $k$ $1$'s. We now define the {\bf content vector} function 
\[
 \content^m : \inc^m(\lambda) \to \{ 0, 1\}^m = \coprod_{0 \leq  n \leq m} \binom{[m]}{n}
 \] 
 by 
\[
\content^m(T) = (c_1, \dots, c_m),
\] 
where $c_i = 1$ if $i \in \mathrm{range}(T)$ and $c_i = 0$ if $i \notin \mathrm{range}(T)$. (Note that this definition of content differs from the usual notion of content for semistandard tableaux in that here we do not care about the multiplicity of a label, but only its presence or absence.)

\begin{example}\label{ex:deflate}
If $T = \ytableaushort{456,125} \in \inc^7(2 \times 3)$, then the deflation of $T$ is \[\deflate^7(T) = \ytableaushort{345,124} \in \incgl^5(2 \times 3).\] The content vector of $T$ is $\content^7(T) = (1,1,0,1,1,1,0) \in \binom{[7]}{5} \subset \{0,1\}^7$.
\end{example}
Now if $\deflate^m(T) \in \incgl^n(\lambda)$, then $\content^m(T) \in \binom{[m]}{n}$. We denote by $\compress^m$ the product map
\[
\compress^m \coloneqq (\deflate^m,\content^m).
\] 
\begin{proposition}\label{prop:compressbijective} The map 
\[
\compress^m : \inc^m(\lambda) \to \coprod_{0 \leq n \leq m} \left( \incgl^n(\lambda) \times \binom{[m]}{n} \right)
\]
 is bijective.
\end{proposition}
\begin{proof}
A two-sided inverse for $\compress^m$ is given as follows. For any positive integer $j$, let $[j]$ denote the set $\{1, 2, \dots, j\}$.
For a binary vector $v \in \{0,1\}^m$, let $N_v$ be the number of $1$'s in $v$, and define a {\bf vector inflation} map \[\inflate^m_v : [N_v] \to [m]\] by
\[ \inflate^m_v(k) = \min \bigg\lbrace n \in [m]:   \sum_{\ell = 1}^n v\langle \ell \rangle = k \bigg\rbrace.\] (We use angled brackets $``\langle \phantom{\ell} \rangle"$ throughout the paper to denote vector components.) An integer $j \in [m]$ is in the range of $\inflate^m_v$ if and only if $v\langle j \rangle = 1$. Therefore
\begin{align*}
[\inflate^m_{\content^m(T)} \circ \deflate^m(T)](\x) &= \min \bigg\lbrace n \in [m]:    \# \{ h \in \mathrm{range}(T): h \leq n \} = [\deflate^m(T)](\x) \bigg\rbrace \\ &= T(\x). 
\end{align*} Now define the {\bf tableau inflation} map 
\[
\tinflate^m : \coprod_{0 \leq n \leq m} \left( \incgl^n(\lambda) \times \binom{[m]}{n} \right) \to \inc^m(\lambda)
\] 
by 
\[
\tinflate^m(S,v) = \inflate^m_v \circ S.
\]
Say $(S,v) \in \incgl^n(\lambda) \times \binom{[m]}{n}$. Since $S$ is surjective onto $[n]$ and $\inflate^m_v$ maps $[n]$ onto the indices of nonzero components in $v$, $\content^m (\tinflate^m(S,v)) = v$.  Also,
\begin{align*}
 [\deflate^m(\inflate^m_v  \circ S)](\x) &= \# \{ h \in \mathrm{range}(\inflate^m_v \circ S): h \leq \inflate^m_v \circ S(\x) \} \\  
 &= \# \{ h: v\langle h\rangle \neq 0 \text{ and } \sum_{\ell = 1}^h v\langle\ell\rangle \leq S(\x)  \} \\
 &= S(\x).
\end{align*}
Therefore, $\tinflate^m$ is a two-sided inverse for $\compress^m$.
\end{proof} 
\begin{example}\label{ex:reinflate}
Let $v = (1,1,0,1,1,1,0) \in \{0,1\}^7$. Then, $N_v = 5$ and the map $\inflate^7_v : [5] \to [7]$ is given by 
\begin{align*}
\inflate^7_v(1) &= 1, \\
\inflate^7_v(2) &= 2, \\
\inflate_v^7(3) &= 4, \\
\inflate_v^7(4) &= 5, \\
\inflate_v^7(5) &= 6. 
\end{align*}
Now, for $S = \ytableaushort{345,124}$, we have $\tinflate^7(S,v) = \inflate^7_v \circ S = \ytableaushort{456,125}$. Note that this process has recovered the tableau $T$ of Example~\ref{ex:deflate} from $S=\deflate^7(T)$ and $v=\content^7(T)$.
\end{example}

\section{Interaction between $K$-promotion and deflation}\label{sec:interaction}
A classical approach to the cohomological Schubert calculus of type $A$ Grassmannians is to use the jeu de taquin on standard Young tableaux introduced by M.-P.~Sch\"utzenberger \cite{Schutzenberger:jdt}. (See \cite{Fulton, Manivel} for modern expositions of the classical jeu de taquin theory.) A theme of modern Schubert calculus has been extending such theories to richer generalized cohomologies and in particular into the $K$-theory ring of algebraic vector bundles. For a partial survey of recent work related to $K$-theoretic Schubert calculus and the associated combinatorics, see \cite{Pechenik.Yong:genomic}.

The first $K$-theoretic Littlewood-Richardson rule was discovered by A.~Buch \cite{Buch}; this rule, however, was not based on jeu de taquin. H.~Thomas and A.~Yong \cite{Thomas.Yong:K} later found a different Littlewood-Richardson rule for the same structure coefficients by directly extending M.-P.~Sch\"utzenberger's jeu de taquin. This latter rule was conjectured in \cite{Thomas.Yong:K} to extend to the $K$-theoretic Schubert structure coefficients of all minuscule varieties, as was proven by a combination of \cite{Buch.Ravikumar,Clifford.Thomas.Yong,Buch.Samuel}. In this Thomas-Yong theory, the role of standard Young tableaux is filled by increasing tableaux.

The $K$-jeu de taquin of H.~Thomas and A.~Yong \cite{Thomas.Yong:K} gives rise to a $K$-promotion operator $\pro^m$ on $\inc^m(\lambda)$ that directly extends M.-P.~Sch\"utzenberger's classical definition of promotion \cite{Schutzenberger:promotion}. $K$-promotion was first studied in \cite{Pechenik}, and has been further investigated in \cite{BPS, Pressey.Stokke.Visentin, Rhoades:skein, DPS, Pechenik:frames,Vorland}. 

In this section, we determine how $\pro^m$ interacts with $\compress^m$. Our results will show that $\pro^m$ is controlled by its action on gapless increasing tableaux. The power of this observation is that there are only finitely many gapless increasing tableaux of any fixed shape $\lambda$. Thus a finite amount of data governs the action of $\pro^m$ on $\inc^m(\lambda)$ for all $m$.

First, we define $\pro^m$, setting up notation that we will need. Instead of using the original definition of \cite{Pechenik}, based on $K$-jeu de taquin, it will be convenient to use an equivalent formulation from \cite{DPS}, based on certain involutions. 
Let $T \in \inc^q(\lambda)$. The {\bf $K$-Bender-Knuth operator} $\kbk_i$ acts on $T$ by swapping the letters $i$ and $i+1$ everywhere in $T$ \emph{where doing so would not violate the increasingness conditions.} More precisely, $\kbk_i$ looks at the set of boxes labeled $i$ or $i+1$, decomposes that set into edge-connected components, swaps the labels $i$ and $i+1$ in each component that is a single box, and acts trivially on each non-trivial connected component. 
Notice that each $\kbk_i$ is an involution.
Using the characterization of \cite[Proposition~2.4]{DPS}, {\bf $K$-promotion} may then be defined by
\[
\pro^q(T) \coloneqq \kbk_{q-1} \circ \cdots \circ \kbk_1(T).
\]
Examples of the action of $K$-Bender-Knuth operators are shown in Figure~\ref{fig:kbk}, while an example of the full process of $K$-promotion is shown in Figure~\ref{fig:promotion}. 

\begin{figure}[h]
\begin{tikzpicture}
\node (A) {\ytableaushort{\none \none 6,\none 45,1235 } };
\node[above right = 1 and 2 of A] (B) {\ytableaushort{\none \none 6, \none 35, 1245}};
\node[ right = 4 of A] (C) {\ytableaushort{\none \none 6, \none 45,1234}};
\node[below right = 1 and 2 of A] (D) {\ytableaushort{\none \none 6, \none 45,1236}};
\path (A) edge[pilpil,shorten >=-0mm,shorten <=-5mm] node[above]{$\kbk_3$} (B);
\path (A) edge[pilpil] node[above]{$\kbk_4$} (C);
\path (A) edge[pilpil,shorten >=-8mm] node[above]{$\kbk_5$} (D);
\end{tikzpicture}
\caption{The three increasing tableaux on the right are obtained from the increasing tableau at left by applying the indicated $K$-Bender-Knuth operators. The arrows point in both directions since each $K$-Bender-Knuth operator is involutive.}\label{fig:kbk}
\end{figure}
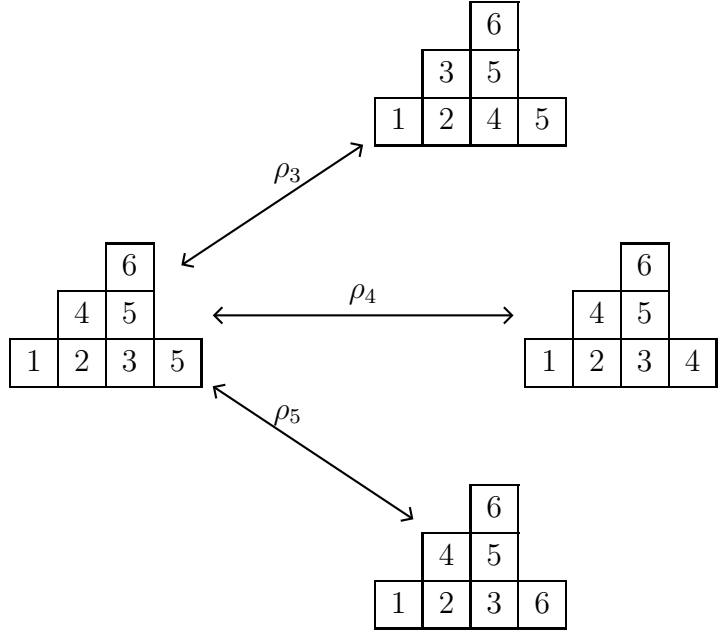

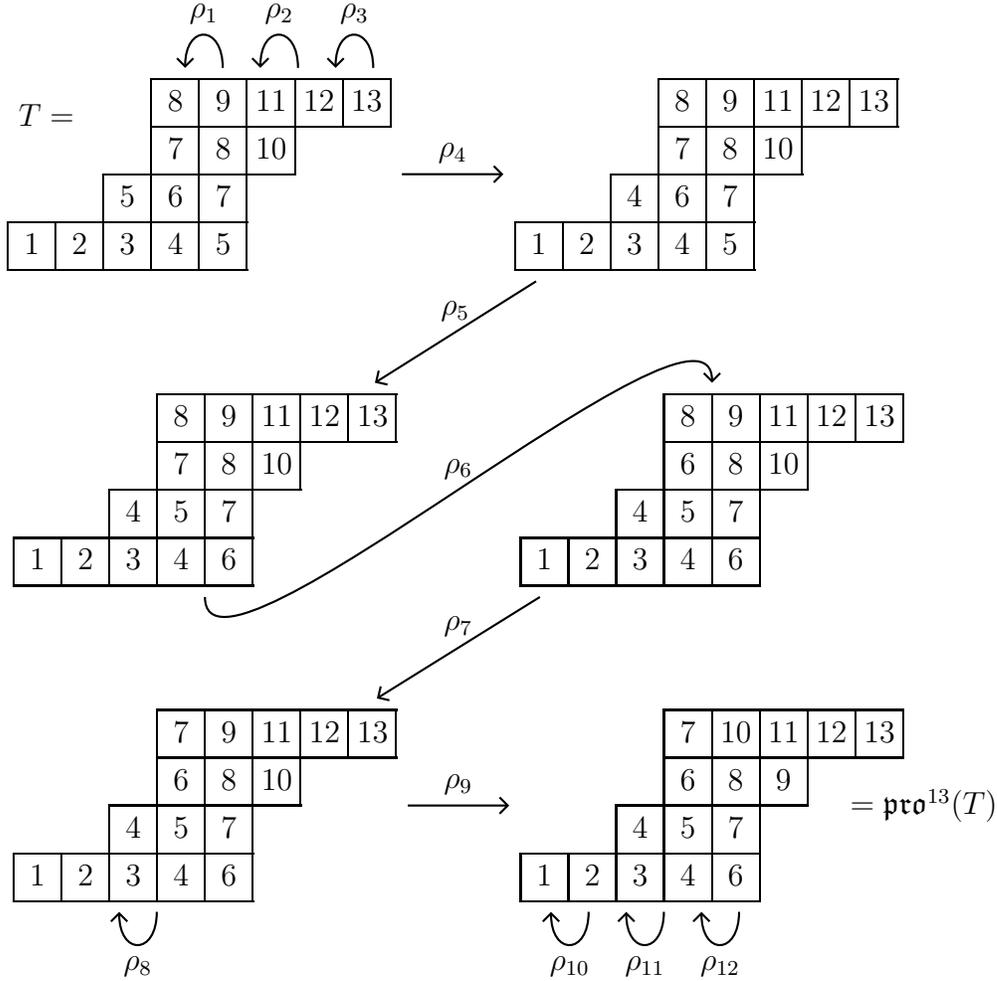
\begin{figure}[h]
\begin{center}
\begin{tikzpicture}
\node (T) {$T= \!\!\!\!\!\!\!$\hspace{-4mm}\ytableaushort{ \none \none \none 89{11}{12}{13},\none \none \none 78{10},\none \none 567, 12345}};
\node[right  = 1.35 of T] (T4) {\ytableaushort{ \none \none \none 89{11}{12}{13}, \none \none \none 78{10}, \none \none 467, 12345}};
\node[below  = 1.35 of T] (T5) {\ytableaushort{   \none \none \none 89{11}{12}{13}, \none \none \none 78{10},\none \none 457,12346}};
\node[right  = 1.35 of T5] (T6) {\ytableaushort{   \none \none \none 89{11}{12}{13}, \none \none \none 68{10},\none \none 457, 12346 }};
\node[below  = 1.35 of T5] (T7) {\ytableaushort{   \none \none \none 79{11}{12}{13},\none \none \none 68{10},\none \none 457, 12346 }};
\node[right  = 1.35 of T7] (T9) {\ytableaushort{  \none \none \none 7{10}{11}{12}{13},\none \none \none 689, \none \none 457, 12346 }};
\node[right = -1 of T9] (finishlabel) {$=\pro^{13}(T)$};
\node[below left = -.3 and -0.7 of T9] (fake1) {};
\node[below left = -.3 and -1.2 of T9] (fake2) {};
\node[below left = -.3 and -1.7 of T9] (fake3) {};
\node[below left = -.3 and -2.2 of T9] (fake4) {};
\node[below left = -.3 and -2.7 of T9] (fake5) {};
\node[below left = -.3 and -3.2 of T9] (fake6) {};
\node[above left = -.3 and -2.5 of T] (fake7) {};
\node[above left = -.3 and -3.0 of T] (fake8) {};
\node[above left = -.3 and -3.5 of T] (fake9) {};
\node[above left = -.3 and -4.0 of T] (fake10) {};
\node[above left = -.3 and -4.5 of T] (fake11) {};
\node[above left = -.3 and -5.0 of T] (fake12) {};
\node[below left = -.3 and -1.7 of T7] (fake13) {};
\node[below left = -.3 and -2.2 of T7] (fake14) {};
\path (T) edge[pil]  node[above]{$\rho_4$} (T4);
\path (T4) edge[pil] node[above]{$\rho_5$} (T5);
\draw[->, thick] (T5.south) .. node[above]{$\rho_6$} controls ([yshift=-3cm] T5) and ([yshift=3cm] T6) .. (T6.north);
\path (T6) edge[pil] node[above]{$\rho_7$} (T7);
\path (fake14) edge[pil,out=270,in=270,looseness=3] node[below]{$\rho_8$} (fake13);
\path (T7) edge[pil] node[above]{$\rho_9$} (T9);
\path (fake2) edge[pil, out=270, in=270, looseness=3] node[below]{$\rho_{10}$} (fake1);
\path (fake4) edge[pil, out=270, in=270, looseness=3] node[below]{$\rho_{11}$} (fake3);
\path (fake6) edge[pil, out=270, in=270, looseness=3] node[below]{$\rho_{12}$} (fake5);
\path (fake8) edge[pil, out=90, in=90, looseness=3] node[above]{$\rho_{1}$} (fake7);
\path (fake10) edge[pil, out=90, in=90, looseness=3] node[above]{$\rho_{2}$} (fake9);
\path (fake12) edge[pil, out=90, in=90, looseness=3] node[above]{$\rho_{3}$} (fake11);
\end{tikzpicture}
\end{center}
\caption{The calculation of the $K$-promotion $\pro^{13}(T)$ of the increasing tableau $T$ (on the Cayley-Moufang poset) through the action of the sequence of $K$-Bender-Knuth operators $\rho_1, \dots, \rho_{12}$.}\label{fig:promotion}
\end{figure}

For the remainder of this section, we assume that $\lambda$ is a miniscule poset. In particular, $\lambda$ has a unique maximal element.
\begin{proposition}\label{prop:deflation_commutation}
Let $T \in \inc^m(\lambda)$. If $\content^m(T) \langle 1 \rangle = 1$, then
\begin{equation}\label{eq:deflation_commutation}
\deflate^m \circ \pro^m(T) = \pro^{m_T} \circ \deflate^m(T).
\end{equation}
If $\content^m(T) \langle 1 \rangle = 0$, then $\pro^m$ decrements each label of $T$ by one. 
\end{proposition} 
Note that, by Proposition~\ref{prop:deflation_commutation}, the behavior of $\pro^m$ on $T \in \inc^m(\lambda)$ is fundamentally different depending on whether or not the label $1$ appears in $T$. This distinction will lead to interestingly complicated counting formulas in Sections~\ref{sec:period} and~\ref{sec:arithmetic}. In order to prove Proposition~\ref{prop:deflation_commutation}, we will first need the following lemma.

\begin{lemma} \label{lem:bullet_placement}
Let $T \in \inc^m(\lambda)$.
Let the ordered set of labels of $T$ be \[ \lbrace i_1 < i_2 < \dots < i_{m_T-1} < i_{m_T} \rbrace.\] (In the notation above, $i_j = \inflate^m_{\content^m(T)}(j)$.) Then, for all  $1 \leq r < m_T$,
\[ \deflate^m \circ \kbk_{i_{r+1}-1} \circ \cdots \circ  \kbk_{i_r}(T) = \kbk_{r} \circ \deflate^m(T). \]
In addition, the ordered set of labels of $\kbk_{i_{r+1}-1} \circ \cdots \circ  \kbk_{i_r}(T)$ is \[ \lbrace i_1< \dots < i_{r-1} < i_{r+1}-1< i_{r+1}< \dots < i_{m_T} \rbrace.\] 
\end{lemma}
\begin{proof} 

Fix $1 \leq r < m_T$. For $s \geq i_r$, we use the shorthand 
\[
T_{s} \coloneqq \kbk_{s-1} \circ \cdots \circ  \kbk_{i_r}(T);
\] in particular, $T_{i_r} = T$. 

Let $R \subseteq \lambda$ be the set of boxes $\z \in \lambda$ such that $T(\z) = i_r$. Consider $\x \in R$.
Either 
\begin{itemize}
\item[(1-A)] $\x$ is covered by at least one box labeled $i_{r+1}$ in $T$ or
\item[(1-B)] it is not.
\end{itemize}  In either case, observe that for $i_r \leq \ell  < i_{r+1}-1$, $\kbk_{\ell}$ acts on $T_\ell$ by changing the label of each $\z \in R$ from $\ell$ to $\ell+1$ and leaving all other labels unchanged, since $T_\ell$ has no boxes labeled $\ell+1$. In the case (1-A),  \[ T_{i_{r+1}}(\x) = \kbk_{i_{r+1}-1}(T_{i_{r+1}-1})(\x) = i_{r+1}-1.\]
In the case (1-B), \[ T_{i_{r+1}}(\x) = \kbk_{i_{r+1}-1}(T_{i_{r+1}-1})(\x) = i_{r+1}.\] 

Now let $\y \in \lambda$ be a box with $T(\y) = i_{r+1}$. Either
\begin{itemize}
\item[(2-A)] $\y$ covers at least one box labeled $i_{r}$ in $T$ or 
\item[(2-B)] it does not.
\end{itemize}
In either case, recall that for $i_r \leq \ell  < i_{r+1}-1$, the label of $\y$ is unchanged by the action of $\kbk_{\ell}$ on $T_\ell$. In the case (2-A), 
 \[ T_{i_{r+1}}(\y) = \kbk_{i_{r+1}-1}(T_{i_{r+1}-1})(\y) = i_{r+1}.\] 
In the case (2-B), 
 \[ T_{i_{r+1}}(\y) = \kbk_{i_{r+1}-1}(T_{i_{r+1}-1})(\y) = i_{r+1}-1.\] 
 
All other boxes are unaffected by $\kbk_\ell$ for $\ell$ in this range. Since there  must be either an $\x \in R$ satisfying (1-A) or a $\y \in \lambda$ satisfying (2-B),  the ordered set of entries of $T_{i_{r+1}}$ is $\lbrace i_1 < \dots < i_{r-1} < i_{r+1}-1 < i_{r+1} < \dots < i_{m_T} \rbrace$, proving the last sentence of the lemma. 

Now therefore, if $\x$ satisfies (1-A), then
\[ \deflate^m (T_{i_{r+1}})(\x) = r, \] 
while if $\x$ satisfies (1-B), then
\[ \deflate^m (T_{i_{r+1}})(\x) = r+1. \] 
Similarly, if $\y$ satisfies (2-A), then
\[ \deflate^m (T_{i_{r+1}})(\y) = r+1, \] 
while if $\y$ satisfies (2-B), 
\[ \deflate^m (T_{i_{r+1}})(\y) = r. \]
It is easily seen that these values match $\kbk_{r} \circ \deflate^m(T)(\x)$ and $\kbk_{r} \circ \deflate^m(T)(\y)$ in each case. Since all boxes not satisfying one of these four cases are unchanged by the $K$-Bender-Knuth operators in question, this completes the proof of the lemma. 
\end{proof}

We are now ready to prove Proposition~\ref{prop:deflation_commutation}. 
\begin{proof}[Proof of Proposition~\ref{prop:deflation_commutation}]
Let the ordered set of labels of $T \in \inc^m(\lambda)$ be \[ \lbrace i_1 < i_2 < \dots < i_{m_T-1} < i_{m_T} \rbrace.\] (That is, $i_j = \inflate^m_{\content^m(T)}(j)$.) 

 Let $v_r$ be the content vector corresponding to the ordered set 
 \[
 \lbrace i_2-1 < i_3-1 < \dots < i_{r} - 1 < i_{r+1}-1 < i_{r+1} < \dots < i_{m_T} \rbrace.
 \]
  That is, $v_r$ is the binary vector of length $m$ with $1$s in exactly these positions. 
 
First, suppose that $\content^m(T) \langle 1 \rangle = 1$, so $i_1 = 1$. Then by Lemma~\ref{lem:bullet_placement}, 
\[
\deflate^m \circ \kbk_{i_{2}-1} \circ \cdots \circ  \kbk_{1}(T) = \kbk_1 \circ \deflate^m(T).
\]
Hence by the proof of Proposition~\ref{prop:compressbijective}, 
\[  \kbk_{i_{2}-1} \circ \cdots \circ  \kbk_{1}(T) = \inflate_{v_1}^m \circ (\kbk_1 \circ \deflate^m(T)),\] since $v_1 = \content^m(\kbk_{i_2 - 1} \circ \cdots \circ \kbk_1(T))$. 
By induction using Lemma~\ref{lem:bullet_placement} and the proof of Proposition~\ref{prop:compressbijective}, we then see that 
\begin{align} 
\kbk_{i_{m_T}-1} \circ \cdots \circ  \kbk_{1}(T) &= \inflate^m_{v_{(m_T-1)}} \circ \kbk_{m_T-1}  \circ \deflate^m \circ \cdots \circ \inflate^m_{v_1} \circ \kbk_1 \circ \deflate^m(T) \notag \\ 
&= \inflate^m_{v_{(m_T-1)}} \circ \kbk_{m_T-1}   \circ  \cdots \circ \kbk_1 \circ \deflate^m(T), \label{eq:defldefl}
\end{align} 
since $v_r = \content^m(\kbk_{i_{r+1} - 1} \circ \cdots \circ \kbk_1(T))$.
If $m = i_{m_T}$, we are done after applying $\deflate^m$ to both sides of Equation~\eqref{eq:defldefl}. Otherwise, Lemma~\ref{lem:bullet_placement} implies that the unique maximal element of $\lambda$ is labeled with $i_{m_T}$ in $\kbk_{i_{m_T}-1} \circ \cdots \circ  \kbk_{1}(T)$, so $\kbk_{i_{m_T}}$ increments the maximal box to $i_{m_T}+1$. By induction, $\kbk_\ell$ increments this same box to $i_{\ell}+1$ for $\ell = m_T,m_T+1, \dots, m-1$. Therefore, \[ \deflate^m \circ \kbk_{m-1} \circ \cdots \circ \kbk_{1} (T) = \deflate^m \circ \kbk_{i_{m_T}-1} \circ \cdots \circ \kbk_{1} (T).\] This completes the proof in the case $\content^m(T) \langle 1 \rangle = 1$. 

Now, suppose that $\content^m(T) \langle 1 \rangle = 0$, so $i_1  > 1$. Then $\kbk_{r}$ acts trivially on $\kbk_{r-1} \circ \cdots \circ \kbk_{1}(T)$ for $r < i_1 -1$. Next, $\kbk_{i_1-1}$ decrements all boxes labeled $i_1$ by one, since there are no adjacent boxes labeled $i_1-1$. But then there are no boxes labeled $i_1$, so $\kbk_{i_1}$ decrements all boxes labeled $i_1+1$ by one. Inductively, it follows that $\kbk_{r}$ decrements boxes labeled $r+1$ by one for $i_1-1 \leq r < m_T$ and leaves other boxes unchanged. It is then easy to see that $\kbk_r$ acts trivially for $r \geq m_T$ and the proposition follows.
\end{proof} 

\section{Computation of period}\label{sec:period} In this section, we use Proposition~\ref{prop:deflation_commutation} to relate the period of $T \in \inc^m(\lambda)$ under $\pro^m$ to data concerning $\deflate^m(T)$. Let 
\[\Sigma^m : \lbrace 0,1\rbrace^m \rightarrow \lbrace 0,1\rbrace^m\]
 be the cyclic rotation defined by 
 \[
 \Sigma^m(v_1, \dots, v_m) = (v_2, \dots, v_m, v_1)
 \]
  for $(v_1, \dots, v_m) \in \lbrace 0,1 \rbrace^m$. 
  
\begin{theorem}\label{thm:periodthm}
Fix $T \in \inc^m(\lambda)$. Let $\tau$ be the the period of $\pro^{m_T}$ on $\deflate^m(T)$ and $\ell$ be the period of $\Sigma^m$ on $\content^m(T)$. Then, the period of $\pro^m$ on $T$ is \[\frac{\ell  \tau}{\mathrm{gcd}(\ell m_T / m,\tau)}. \]
\end{theorem} 
    
  \begin{proof}
Define
  \[
  K^m: \coprod_{n \leq m}\incgl^n(\lambda) \times \binom{[m]}{n} \rightarrow \coprod_{n \leq m}\incgl^n(\lambda) \times \binom{[m]}{n}
  \] by
\[
K^m(S,v) =
\begin{cases}
    \big( \pro^{m_S}(S),\Sigma^m(v) \big),  & \text{if } v\langle1\rangle = 1; \\        
   \big( S,\Sigma^m(v) \big), & \text{otherwise.}
\end{cases}
\]
We first show that, for $T \in \inc^m(\lambda)$,
\begin{equation}\label{eq:k_commutes}
\compress^m \circ \pro^m(T) = K^m \circ \compress^m(T).
\end{equation}
Let $\pi_1$ be the projection onto the first factor of $\coprod_{n \leq m}\incgl^n(\lambda) \times \binom{[m]}{n}$ and $\pi_2$ the projection onto the second factor. It is immediate from the definitions that \[ \pi_2(K^m \circ \compress^m(T)) = \Sigma^m \circ \content^m(T).\] But by \cite[Lemma~2.1]{DPS},
\[  \Sigma^m \circ \content^m(T) = \content^m \circ \pro^m(T) = \pi_2(\compress^m \circ \pro^m(T)).\] Thus, the two sides of Equation~(\ref{eq:k_commutes}) are equal in the second factor.  

We now show that Equation~(\ref{eq:k_commutes}) holds in the first factor. Let \[ w \coloneqq \content^m(T) = \pi_2(\compress^m(T)). \]  By the proof of Proposition~\ref{prop:compressbijective}, we can write
\[ T = \tinflate^m(\deflate^m(T), w).\] 
If $w\langle 1 \rangle = 1 $, then
\begin{align*}
\pi_1(\compress^m \circ \pro^m(T)) &= \deflate^m \circ \pro^m(T) \\
&= \pro^{m_T} \circ \deflate^m (T) &\text{(by Proposition~\ref{prop:deflation_commutation}, $w\langle 1 \rangle = 1$)} \\
&= \pi_1(K^m \circ \compress^m(T)) &\text{(because $w\langle 1 \rangle = 1$).}
\end{align*}
On the other hand, if $w \langle 1 \rangle = 0$, then 
\begin{align*}
\pi_1(\compress^m \circ \pro^m(T)) &= \deflate^m \circ \pro^m(T) \\
&= \deflate^m(T) &\text{(by Proposition~\ref{prop:deflation_commutation}, $w \langle 1 \rangle = 0$)} \\
&= \pi_1 ( K^m \circ \compress^m(T)) &\text{(because $w\langle 1 \rangle = 0$).}
\end{align*}
This completes the proof of Equation~(\ref{eq:k_commutes}).

Now, since $\compress^m$ is a bijection by Proposition~\ref{prop:compressbijective}, $(\pro^m)^{\circ n}(T) = T$ exactly when $\compress^m \circ (\pro^m)^{\circ n}(T) = \compress^m \circ T$. So, by Equation~(\ref{eq:k_commutes}), we are reduced to determining which powers $n$ of $K^m$ stabilize $\compress^m(T)$.

Say $(K^m)^{\circ n}(\compress^m(T)) = \compress^m(T)$. Applying $\pi_2$ to both sides, we have \[ (\Sigma^m)^{\circ n}(\content^m(T)) = \content^m(T), \] so $n$ is a multiple of $\ell$. Hence, $n = t \ell$ for some integer $t \in \mathbb{Z}$.  Applying instead $\pi_1$ to both sides, we have \[ (\pro^m)^{\circ n'}(\deflate^m(T)) = \deflate^m(T), \] where 
\begin{align}\label{eq:n'}
n' &\coloneqq n - \# \{ 0 \leq j \leq n-1 : \pi_2 ((K^m)^{\circ j} \circ \compress^m(T)) \langle 1 \rangle = 0\} \\
&= n - \#  \{ 1 \leq i \leq n : \content^m(T) \langle i \mod m \rangle = 0\}. \nonumber
\end{align}
Let $r$ be the number of $0$s among the first $\ell$ entries of $\content^m(T)$. Then, $n' = t(\ell - r)$. Since $(\pro^m)^{\circ n'}$ fixes $\deflate^m(T)$, $n'$ is a multiple of $\tau$. Hence, $n' = s \tau$ for some integer $s$, and so $n = n' + tr = s\tau + tr$.

Thus, we have $(K^m)^{\circ n}  \circ \compress^m(T) = \compress^m(T)$ if and only if $n = t \ell = s \tau + t  r$ for some integers $s$ and $t$. Hence, the period of $\pro^m$ on $T$ is the least positive integer $h$ that can be simultaneously written in the forms $h = t \ell$ and $h=s \tau + t r$ for the same value of $t$. But $t \ell = s\tau + tr$ implies $t(\ell-r) = s\tau$. Since $0 \leq r < \ell$, the least such $h$ is clearly achieved when $t$ is minimal, i.e.\ when $t(\ell-r)$ is the least common multiple of $\ell-r$ and $\tau$. Thus, $h$ can be expressed as 
\begin{equation}\label{eq:period}
h = t \ell = \frac{\ell \, \text{lcm}(\ell-r,\tau)}{\ell-r} = \frac{\ell \tau}{\text{gcd}(\ell-r,\tau)}.
\end{equation}

Finally, observe that $r = \frac{\ell (m-m_T)}{m}$, so that \[ \ell - r = \frac{\ell m_T}{m}.\]
Therefore,
\[ h = \frac{\ell \tau}{\text{gcd}(\frac{\ell m_T}{m},\tau)}, \]
as desired. 
\end{proof}

The following corollary will allow us to determine the period of $\pro^m$ on $\inc^m(\lambda)$ in Section~\ref{sec:arithmetic}. 

\begin{corollary}\label{corr:pdbound}
In the notation of Theorem~\ref{thm:periodthm}, suppose $j$ is a positive integer such that $\tau$ divides $j m_T$. Then, the period $h$ of $\pro^m$ on $T$ divides $j m$. 
\end{corollary} 
\begin{proof}
First, say $j m_T = \tau$. Then,
\begin{equation}\label{eq:pdbound} h=  \frac{\ell \tau}{\text{gcd}(\frac{\ell m_T}{m},\tau)} = \frac{\ell \cdot j m_T}{\text{gcd}(\frac{\ell m_T}{m},j m_T)} = \frac{\ell \cdot j m_T}{\frac{\ell m_T}{m}} = jm. 
\end{equation}
Now, say $\tau$ divides $jm_T$. Then, \[ \frac{\tau}{\text{gcd}(\frac{\ell m_T}{m},\tau)} \ \ \  \text{    divides   } \ \ \ \frac{j m_T}{\text{gcd}(\frac{\ell m_T}{m},j m_T)}.\] Thus, $h$ divides $j m$. 
\end{proof}

\section{Proof of Theorems~\ref{thm:exceptionals} and~\ref{thm:propeller}}\label{sec:arithmetic}

In this final section, we collect the above results into proofs of Theorems~\ref{thm:exceptionals} and~\ref{thm:propeller}. Consider a poset $P$ that is either the Cayley-Moufang poset, the Freudenthal poset, or one of the propellers. By Corollary~\ref{cor:multisets},  the multiset of $\Psi$-orbit cardinalities for $J(P \times \mathbf{k})$ equals the multiset of $\pro^{\rank(P)+ k+1}$-orbit cardinalities for $\inc^{ \rank(P)+ k+1}(P)$. Therefore, we may prove Theorems~\ref{thm:exceptionals} and~\ref{thm:propeller} by studying increasing tableaux instead of plane partitions. 

By Proposition~\ref{prop:deflation_commutation}, the $\pro$-orbit structure of increasing tableaux is controlled by the data of gapless increasing tableaux and binary vectors. But for any fixed $P$, $\incgl(P)$ is a finite set. Using a computer, we found all $549$ gapless increasing tableaux of shape $P_{CM}$, as well as all $624\, 493$ gapless increasing tableaux of shape $P_F$. For each such tableau $T \in \incgl^m(P_{CM})$ or $\incgl^m(P_F)$, we then determined its period under $\pro^m$ by direct calculation. These statistics are given in Table~\ref{tab:E6} for $P_{CM}$ and in Table \ref{tab:E7} for $P_F$. We used \textsc{SageMath}~\cite{sage, combinat} for these calculations.

\begin{table}[ht]
\begin{tabular}{|c|c|c|}
\hline
${m_T}$ & period ($\tau$) & number of orbits ($N$)\\
  \hline
  11 & 1 & 1\\
  \hline
  12 & 3 & 1\\ \cline{2-3}
   & 12 & 1 \\
   \hline
  13 & 13 & 6\\
  \hline
  14 & 7 & 2\\\cline{2-3}
  & 14 & 12\\
  \hline
    15 & 15 & 13\\
  \hline
  16 & 2 & 1\\\cline{2-3}
   & 4 & 1\\\cline{2-3}
   & 8 & 1\\\cline{2-3}
  & 16 & 4\\
  \hline
\end{tabular}
\caption{The distribution of $\pro^{m_T}$-orbits of gapless increasing tableaux in $\incgl^{m_T}(P_{CM})$ for each ${m_T}$.}
\label{tab:E6}
\end{table}

\begin{table}[ht]
\begin{tabular}{|c|c|c|}
\hline
${m_T}$ & period ($\tau$) & number of orbits ($N$) \\
  \hline
  17 & 1 & 1\\
  \hline
  18 & 2 & 1\\ \cline{2-3}
   & 18 & 2 \\
   \hline
  19 & 19 & 30\\
  \hline
  20 & 20 & 228\\
  \hline
    21 & 7 & 3\\ \cline{2-3}
    & 21 & 1044 \\
  \hline
  22 & 22 & 3053\\\cline{2-3}
   & 66 & 2\\
  \hline
  23 & 23 & 5813\\ \cline{2-3}
  & 69 & 13 \\
   \hline
  24 & 8 & 7\\\cline{2-3}
   & 24 & 7195\\\cline{2-3}
   & 48 & 4\\\cline{2-3}
   & 72 & 26\\
   \hline
  25 & 25 & 5602\\ \cline{2-3}
  & 50 & 8 \\ \cline{2-3}
  & 75 & 21 \\
   \hline
  26 & 2 & 2\\ \cline{2-3}
  & 26 & 2495 \\ \cline{2-3}
  & 52 & 4 \\ \cline{2-3}
  & 78 & 6 \\
   \hline
 27 & 3 & 2\\ \cline{2-3}
  & 9 & 4 \\ \cline{2-3}
  & 27 & 484 \\
  \hline
\end{tabular}
\caption{The distribution of $\pro^{m_T}$-orbits of gapless increasing tableaux in $\incgl^{m_T}(P_F)$ for each ${m_T}$.}
\label{tab:E7}
\end{table}

It is essentially trivial to observe that, for any $p$, $\incgl(P_p)$ consists of exactly three increasing tableaux. There are two tableaux in $\incgl^{2p}(P_p)$, forming a single $\pro^{2p}$-orbit, and a single tableau in $\incgl^{2p-1}(P_p)$, which is necessarily fixed by $\pro^{2p-1}$. These three tableaux and their orbits are illustrated in Figure~\ref{fig:propeller_orbits} in the case $p=4$. Table~\ref{tab:prop} records the observations of this paragraph.

\begin{figure}[h]
\begin{tikzpicture}
\node (T7) {\ytableaushort{\none \none 4567,1234}};
\node[above right = -0.7 and .2 of T7] (fake1) {};
\node[above right = -1.2 and .2 of T7] (fake2) {};
\node[below left = 2 and 0.6 of T7] (T8A) {\ytableaushort{\none \none 5678,1234}};
\node[right = 4 of T8A] (T8B) {\ytableaushort{\none \none 4678,1235}};
\path (T8A) edge[pil, out=20,in=160,shorten >=-6mm] node[above]{$\pro^8$} (T8B);
\path (T8B) edge[pil, out = 200, in=340,shorten >=-6mm] node[below]{$\pro^8$} (T8A);
\path (fake1) edge[pil, out = 20, in = 340, looseness=4] node[right]{$\pro^7$} (fake2);
\end{tikzpicture}
\caption{The set $\incgl(P_4)$ consists of the three illustrated gapless increasing tableaux. The unique element of $\incgl^7(P_4)$ forms a singleton $\pro^7$-orbit, while $\pro^8$ switches the two elements of $\incgl^8(P_4)$, as shown. The situation for $p \neq 4$ is exactly analogous.}\label{fig:propeller_orbits}
\end{figure}
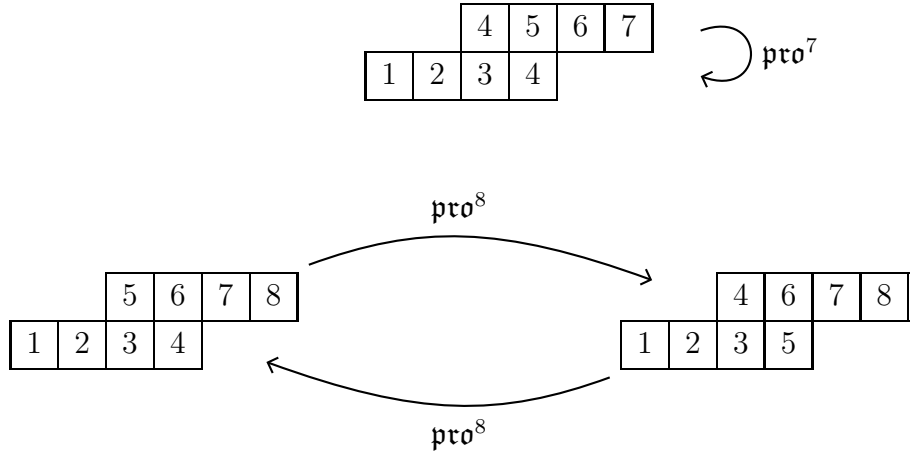

\begin{table}[ht]
\begin{tabular}{|c|c|c|}
\hline
$m_T$ & period ($\tau$) & number of orbits ($N$)\\
  \hline
  $2p-1$ & 1 & 1\\
  \hline
  $2p$ & 2 & 1\\ 
  \hline
\end{tabular}
\caption{The distribution of $\pro^{m_T}$-orbits of gapless increasing tableaux in $\incgl^{m_T}(P_p)$ for each ${m_T}$.}
\label{tab:prop}
\end{table}

We will use Theorems~\ref{thm:periodthm} and~\ref{thm:rsw}, together with the data of Tables~\ref{tab:E6},~\ref{tab:E7}, and~\ref{tab:prop}, to determine the multiset of $\pro^m$-orbit cardinalities on set $\inc^m(P)$ for $P \in \{P_p, P_{CM}, P_F \}$ and any $m$. These tables and Corollary~\ref{corr:pdbound} immediately give the period of $\pro^m$.

\pagebreak

\begin{theorem}\label{thm:actualpdbound}
For $m \gg 0$, the period of $\pro^m$ on $\inc^m(P)$ is 
\begin{itemize}
\item $m$ for $P = P_p$, 
\item $m$ for $P = P_{CM}$, and 
\item $3m$ for $P = P_F$. 
\end{itemize}
(Here `$m \gg 0$' means $m \geq 2p$ for $P = P_p$, $m \geq 12$ for $P = P_{CM}$, and $m \geq 22$ for $P = P_F$.) For $m \in \lbrace 18,19,20, 21 \rbrace$, the period of $\pro^m$ on $\inc^m(P_F)$ is $m$.
\end{theorem}
\begin{proof}
Inspecting Table~\ref{tab:E6} shows that, for $P_{CM}$, the period of each gapless tableau  divides its height ($m_T$). By Corollary~\ref{corr:pdbound}, this implies that for all $m$, the period of each tableau of height $m$ divides $m$. Clearly the period of $\pro^{12}$ on $\inc^{12}(P_{CM})$ is no less than $12$, and whenever $m> 12$, it is possible to find a tableau attaining period $m$ by taking a gapless tableau $T \in \incgl^{12}(P_{CM})$ with period $12$ and then inflating according to a content vector with $\Sigma^m$-period $m$ (for example, by just considering $T \in \incgl^{12}(P_{CM})$ as a gappy tableau in $\inc^m(P_{CM})$). The analogous calculation for $P = P_p$ is easy by inspection of Table~\ref{tab:prop}.

The argument is identical for $P = P_F$ when $m \in \lbrace 18,19,20, 21 \rbrace$, but when $m \geq 22$, we do not have that the period of each gapless tableau divides its height. However, inspecting Table~\ref{tab:E7} shows that, for $P_F$ and for every $m$, the period of each gapless tableaux of height $m_T$ divides $3m_T$. Therefore, Corollary~\ref{corr:pdbound} gives that for every $m$, the period of each tableau of height $m$ divides $3m$. Clearly the period of $\pro^{22}$ on $\inc^{22}(P_{F})$ is no less than $66$, and when $m > 22$, it is possible to find a tableau attaining period $3m$ by taking a gapless tableau  $T \in \incgl^{22}(P_F)$ with period $66$ and inflating according to a content vector of $\Sigma^m$-period $m$ (for example, by considering $T$ as a gappy tableau in $\inc^m(P_F)$).
\end{proof}

\begin{remark}
In light of the fact that one does not necessarily have $(\pro^m)^{\circ m}(T) = T$ for $T \in \inc^m(P_F)$, it is reasonable to ask: 
\begin{quote}
``For which $\x \in P_F$, are we guaranteed $(\pro^m)^{\circ m}(T)(\x) = T(\x)$?''
\end{quote}
 Consider the maximal order ideal $\mathscr{T}$ of $P_F$ that is a tree (in the sense that each $\x \in \mathscr{T}$ covers at most one $\y \in \mathscr{T}$) and the maximal order filter $\mathscr{T}^\vee$ that is a dual tree (in the sense that each $\x \in \mathscr{T}^\vee$ is covered by at most one $\y \in \mathscr{T}^\vee$); let ${\sf Frame}(P_F) \coloneqq \mathscr{T} \cup \mathscr{T}^\vee$. The answer to the question, by directly checking all gapless increasing tableaux of shape $P_F$, is that this property holds precisely for those $\x \in {\sf Frame}(P_F)$. This observation is a Freudenthal analogue of \cite[Theorem~2]{Pechenik:frames} and seems related to \cite[Conjecture~5.4]{Ilango.Pechenik.Zlatin}.
\end{remark}

In order to establish our cyclic sieving results, we will need the fact that the number of elements of $\binom{[i]}{j}$ of any fixed $\Sigma^i$-period is given explicitly by \cite[Theorem~1.1(b)]{Reiner.Stanton.White}:
\begin{theorem}[Reiner, Stanton, and White]\label{thm:rsw}
Let 
\begin{equation*}
f_{i,j}(q) \coloneqq 
\begin{cases}
 \frac{[i]!_q}{[j]!_q\cdot [i-j]!_q} & \text{if } i \geq j \\
 0 & \text{if } i < j 
\end{cases}
\end{equation*}
where $[\ell]!_q \coloneqq \prod_{a=1}^\ell [a]_q$ and $[a]_q \coloneqq 1 + q + \dots + q^{a-1}  = \frac{1-q^{a}}{1-q}$ are the standard $q$-analogues. Then, 
\[\#\left\{ v \in \binom{[i]}{j} : (\Sigma^i)^{\circ s}(v) = v \right\} = f_{i,j}(\zeta^s),\] where $\zeta$ is any primitive $i$th root of unity. \qed
\end{theorem} 

For specified $P \in \{P_p, P_{CM}, P_F\}$, we denote by $m_T\langle i\rangle$, $\tau\langle i\rangle$, and $N\langle i\rangle$ the $i$th element of the first, second, and third columns of the corresponding table, respectively. 

\begin{theorem}\label{thm:mainresult}
Fix $P \in \lbrace  P_p, P_{CM} \rbrace$, $m \gg 0$, and $d$ dividing $m$. Let $\mathcal{R}(P,m,d)$ be the number of increasing tableaux $T \in \inc^m(P)$ whose $\pro^m$-period divides $m/d$. Then,
\begin{equation}\label{eq:mainresulteq} 
\mathcal{R}(P,m,d)  = \sum \limits_{i: \, d \, \vert \, \frac{m_T\langle i\rangle }{\tau\langle i\rangle }} \tau\langle i\rangle \, N\langle i\rangle \, f_{m,m_T\langle i\rangle }(\zeta^{m/d}).
\end{equation}
\end{theorem}
\begin{proof}
Proposition~\ref{prop:compressbijective} gives that a tableau $T \in \inc^m(P)$ corresponds bijectively to a pair consisting of a gapless tableau $U \in \incgl^{m_T}(P)$ and a content vector of length $m$ with $m_T$ ones. The table corresponding to $P$ enumerates all gapless tableaux. We iterate over the rows of the table and determine, for each row, how many content vectors yield tableaux in $\inc^m(P)$ whose periods divide $m/d$. 

For each $i$, let $\ell\langle i\rangle = m_T\langle i\rangle/\tau\langle i\rangle$. Recall the definition of $d'$ from Equation~\eqref{eq:n'}. By Theorem~\ref{thm:periodthm}, a gapless tableau $U \in \incgl^{m_T\langle i\rangle}(P)$ and a content vector $v$ with period $m\langle i\rangle/d'$ inflate to a tableau $V$ of period 
\[ 
\frac{\frac{m\langle i\rangle}{d'} \frac{m_T\langle i\rangle }{\ell\langle i\rangle }}{\gcd(\frac{m_T\langle i\rangle}{d'},\frac{m_T\langle i\rangle }{\ell\langle i\rangle})} = \frac{\frac{m\langle i\rangle}{d'} \frac{m_T\langle i\rangle }{\ell\langle i\rangle}}{\frac{m_T\langle i\rangle \, \gcd(\ell\langle i\rangle,d')}{d' \ell\langle i\rangle}} = \frac{m\langle i\rangle}{\gcd(\ell\langle i\rangle,d')}. 
\] 
Therefore, $V$ has period dividing $m\langle i\rangle/d$ if and only if $d$ divides both $d'$ and $\ell\langle i\rangle$. Now, $d$ divides $d'$ if and only if the period of $v$ divides $m\langle i\rangle/d$. But, by Theorem~\ref{thm:rsw}, the number of content vectors of length $m\langle i\rangle$ with $m_T\langle i\rangle$ $1$s and with period dividing $m\langle i\rangle/d$ is precisely $f_{m,m_T\langle i\rangle}(\zeta^{m/d})$.
\end{proof}

The formula resulting from Theorem~\ref{thm:mainresult} is simple enough to check by hand. 
As in \cite[Proof of Theorem 7.1]{Reiner.Stanton.White}, we need the following elementary identity, which allows us to work with integers rather than $q$-integers.

\begin{lemma}\label{lem:evalq}
Let $\zeta$ be a primitive $N$th root of unity and let $d > 0$ divide $N$. Then $[n]_{\zeta^{N/d}} = 0$ if and only if $d > 1$ and $n \equiv 0 \mod d$. Moreover, if $n_1 \equiv n_2 \mod d$, then 
\begin{equation}\label{eq:evalq}
\pushQED{\qed}
\lim \limits_{q \rightarrow \zeta^{N/d}} \frac{[n_1]_q}{[n_2]_q} = 
\begin{cases}
\frac{n_1}{n_2}, &\text{if }  n_1 \equiv n_2 \equiv 0 \mod d; \\
1, &\text{if } n_1 \equiv n_2 \neq 0 \mod d.
\end{cases} \qedhere \popQED
\end{equation}
\end{lemma}

For example, we evaluate $f^{m-2p+1}_{P_p}(\zeta^{m/d})$ for some $d > 1$ dividing both $m$ and $p$. Since the minuscule posets are Gaussian, as discussed in Remark~\ref{rem:gaussian}, we have
\tiny
\[   f^{m-2p+1}_{P_p}(q) = \frac{[m-(2p-2)]_{q}[m-(2p-2)+1]_{q} \cdots [m-(p-1)]_{q}^2 \cdots [m-1]_{q}[m]_{q}}{[1]_{q}[2]_{q} \cdots [p]_{q}^2 \cdots [2p-2]_{q}[2p-1]_{q}}. \]
\normalsize
Since $f^{m-2p+1}_{P_p}(q)$ is, by definition, a polynomial, we clearly have $f^{m-2p+1}_{P_p}(\zeta^{m/d}) = \lim_{q \rightarrow \zeta^{m/d}} f^{m-2p+1}_{P_p}(q)$.
Thus, Lemma~\ref{lem:evalq} allows us to replace this ratio of polynomials with a ratio of integers by matching equivalence classes modulo $d$ in the numerator and denominator. We see that the numerator and denominator have the same multiset of equivalence classes modulo $d$. Pairing equivalent terms and using Lemma~\ref{lem:evalq}, we have
\[ f^{m-2p+1}_{P_p}(\zeta^{m/d}) = \frac{(m-2p+d)(m-2p+2d) \cdots  (m-d)(m)}{(d) (2d)\cdots  (2p-2d)(2p-d) p} = 2 \binom{m/d}{2p/d}.\]     

An important special case of Lemma~\ref{lem:evalq} allows us to extend Theorem~\ref{thm:rsw}. We have that for any $i, j \in \mathbb{Z}_{> 0}$, if $\zeta$ is a primitive $i$th root of unity and $d$ divides $\gcd(i,j)$, then
\begin{equation}\label{eq:evalstraightshape}
f_{i,j}(\zeta^{i/d}) = \binom{i/d}{j/d} = \frac{(i-(j-d))(i-(j-2d)) \cdots  (i-d)(i)}{(d)  (2d)  \cdots  (j-d) (j)}. 
\end{equation}
Equation~(\ref{eq:evalstraightshape}) holds even when $j > i$, since in that case one of terms in the numerator is $0$.

\begin{proof}[Proof of Theorem~\ref{thm:propeller}] Fix positive integers $p$ and $k$. Let $m = k+2p-1 = k + \rank(P_p) + 1$ and suppose $d$ divides $m$. We have that \[ f^{m-2p+1}_{P_p}(q) = \frac{[m]!_q [m-(p-1)]_q}{[2p-1]!_q [m-(2p-1)]!_q [p]_q }.\]  If $d = 1$, then by Theorem~\ref{thm:mainresult}, 
\begin{align*}
 \mathcal{R}(P_p,m,d) = f_{m,2p-1}(1) + 2 \, f_{m,2p}(1) &= \binom{m}{2p-1} + 2 \binom{m}{2p} \\
 &= \frac{(2m-2p+2)m!}{(2p)!(m-2p+1)!} = f^{m-2p+1}_{P_p}(1),
 \end{align*}
 as desired.
Now, for $d > 1$, we can have that $d$ divides $p$, or $d$ divides $2p-1$, or $d$ divides neither. If $d$ divides $p$, then by Theorem~\ref{thm:mainresult} and Equation~(\ref{eq:evalstraightshape}),
\[ \mathcal{R}(P_p,m,d) = 2 \, f_{m,2p}(\zeta^{m/d}) = 2\binom{m/d}{2p/d} = f^{m-2p+1}_{P_p}(\zeta^{m/d}).\]
If instead $d$ divides $2p-1$, then we have
\[ \mathcal{R}(P_p,m,d) =  f_{m,2p-1}(\zeta^{m/d}) = \binom{m/d}{(2p-1)/d} = f^{m-2p+1}_{P_p}(\zeta^{m/d}).\]
Finally, if $d$ divides neither $p$ nor $2p-1$, then on the one hand Theorem~\ref{thm:mainresult} claims $\mathcal{R}(P_p,m,d) = 0$, while on the other hand $d$ divides $\lceil{\frac{2p-1}{d}}\rceil$ of the terms in the numerator of $f^{m-2p+1}_{P_p}(\zeta^{m/d})$ and at most $\lfloor{\frac{2p-1}{d}}\rfloor$ of the terms in the denominator, so $f^{m-2p+1}_{P_p}(\zeta^{m/d}) = 0$. 
\end{proof}

The verification for $P = P_{CM}$ is similarly straightforward. Before carrying out this verification, we resolve Conjecture~\ref{conj:rush.shi} as it applies to $P_F$.

\begin{theorem}\label{thm:F_bad}
Conjecture~\ref{conj:rush.shi} holds for $P_F$ only when $k \leq 4$.
\end{theorem}
\begin{proof}
That Conjecture~\ref{conj:rush.shi} holds for $P_F$ in the case $k \leq 4$ can be checked numerically using Table~\ref{tab:E7} and the proof of Theorem~\ref{thm:mainresult}.

It remains to show that the conjecture fails for $k \geq 5$. Hence, let $m \geq 22 = 5 +  \rank(P_F) + 1$. By Theorem~\ref{thm:actualpdbound}, $\pro^m$ has period $3m$ on $\inc^m(P_F)$. 

Suppose $T \in \inc^m(P_F)$ were fixed by $\pro^m$.  By Equation~\ref{eq:period}, the $\Sigma^m$-period of $\content^m(T)$ is $1$. Since $\content^m(T)$ is certainly not a vector of all $0$'s, it must therefore be a vector of all $1$'s. Hence, $m_T = m$ and $\deflate(T) = T$. Therefore, $T$ is gapless. However, Table~\ref{tab:E7} shows that no element of $\incgl^m(P_F)$ is fixed by $\pro^m$, for any $m \geq 22$, so such a $T$ does not exist.

Now, let $\zeta$ be a primitive $(3m)$th root of unity. By the above, Conjecture~\ref{conj:rush.shi} claims that $f_{P_F}^{m-17}(\zeta) = 0$. But this is impossible, for then the minimal polynomial of $\zeta$ would divide the numerator of $f^{m-17}_{P_F}(q)$, which itself divides a product of factors of the form $1-q^n$ where $n < 3m$. Thus, Conjecture~\ref{conj:rush.shi} fails in these cases. 
\end{proof}

Finally, we are prepared to prove Theorem~\ref{thm:exceptionals}, the main result of this paper.

\begin{proof}[Proof of Theorem~\ref{thm:exceptionals}]
Theorem~\ref{thm:F_bad} proves the $P_F$ cases. Hence, it remains to consider $P = P_{CM}$.
Fix $k$ and let $m \coloneqq k +11   = k + \rank(P_{CM}) + 1$. 

By inspection of Table~\ref{tab:E6}, the possible values of $\ell\langle i\rangle \coloneqq m_T\langle i\rangle/\tau\langle i\rangle$ are $1$, $11$, $2$, $4$, and $8$. Therefore, we will determine $\mathcal{R}(P_{CM},m,d)$ for each $d$ that divides at least one of these values. We will verify that this number matches the prediction given by $f^{m-11}_{CM}$ and Conjecture~\ref{conj:rush.shi}. For all other values of $d$, we have $\mathcal{R}(P_{CM},m,d) = 0$, so our final check will be that  $f^{m-11}_{CM}(\zeta^{m/d}) = 0$ in these cases.

\smallskip

\noindent
{\sf (Case 1 : $d = 1$)}:
Here, Theorem~\ref{thm:mainresult} and Table~\ref{tab:E6} give
\tiny
\begin{align*}
\mathcal{R}(P_{CM},m,1) &= 1 \cdot 1 \cdot \binom{m}{11} + (3 \cdot 1  + 12 \cdot 1)  \cdot \binom{m}{12} + 13 \cdot 6 \cdot \binom{m}{13} + (7\cdot 2 + 14 \cdot 12) \binom{m}{14}  \\ &\ \ \ \ \ \ + 15 \cdot 13 \cdot \binom{m}{15} + (2 \cdot 1 + 4 \cdot 1 + 8 \cdot 1 + 16 \cdot 4) \binom{m}{16} \\
&= \frac{(m-10)(m-9)(m-8)(m-7)^2(m-6)^2(m-5)^2(m-4)^2(m-3)^2(m-2)(m-1)m}{1\cdot 2\cdot 3\cdot 4^2 \cdot 5^2 \cdot 6^2 \cdot 7^2 \cdot 8^2 \cdot 9 \cdot 10 \cdot 11} \\
&= f_{CM}^{m-11}(1),
\end{align*}
\normalsize
as desired.

\smallskip
\noindent
{\sf (Case 2 : $d = 11$)}:
By Theorem~\ref{thm:mainresult}, Lemma~\ref{lem:evalq}, and Table~\ref{tab:E6}, we have
\begin{align*}
\mathcal{R}(P_{CM},m,11) &= 1 \cdot 1 \cdot \binom{m/11}{11/11} \\
&= m/11 \\ 
&= f_{CM}^{m-11}(\zeta^{m/11}).
\end{align*}

\smallskip
\noindent
{\sf (Case 3 : $d = 8$)}:
By Theorem~\ref{thm:mainresult}, Lemma~\ref{lem:evalq}, and Table~\ref{tab:E6}, we have
\begin{align*}
\mathcal{R}(P_{CM},m,8) &= 2 \cdot 1 \cdot \binom{m/8}{16/8} \\
&= \frac{(m)(m-8)}{8^2} \\
&= f_{CM}^{m-11}(\zeta^{m/8}).
\end{align*}

\smallskip
\noindent
{\sf (Case 4 : $d = 4$)}:
By Theorem~\ref{thm:mainresult}, Lemma~\ref{lem:evalq}, and Table~\ref{tab:E6}, we have
\begin{align*}
\mathcal{R}(P_{CM},m,4) &= 3 \cdot 1 \cdot \binom{m/4}{12/4} + (4 \cdot 1 + 2 \cdot 1) \cdot \binom{m/4}{16/4}  \\
&= \frac{(m-8)(m-4)^2m}{4^2 \cdot 8^2} \\
&= f_{CM}^{m-11}(\zeta^{m/4}).
\end{align*}

\smallskip
\noindent
{\sf (Case 5 : $d = 2$)}:
By Theorem~\ref{thm:mainresult}, Lemma~\ref{lem:evalq}, and Table~\ref{tab:E6}, we have
\begin{align*}
\mathcal{R}(P_{CM},m,2) &= 3 \cdot 1 \cdot \binom{m/2}{12/2} + 7 \cdot 2 \cdot \binom{m/2}{14/2} + (2 \cdot 1 + 4 \cdot 1 + 8 \cdot 1) \cdot \binom{m/2}{16/2} \\
&= \frac{(m-10)(m-8)(m-6)^2(m-4)^2(m-2)m}{10 \cdot 8^2 \cdot 6^2 \cdot  4^2 \cdot 2}\\
&= f_{CM}^{m-11}(\zeta^{m/2}).
\end{align*}

Finally, it remains to check that for all other values of $d$, we have $f^{m-11}_{CM}(\zeta^{m/d}) = 0$.  First, observe that $f^{m-11}_{CM}(\zeta^{m/d}) = 0$ if $d > 11$, for in this case, Lemma~\ref{lem:evalq} implies that the factor $[m]_{\zeta^{m/d}}$ in the numerator is zero, while all of the factors in the denominator are nonzero. 
For the remaining cases $d \in \{3,5,6,7,9,10 \}$, it is similarly easy to check, using Lemma~\ref{lem:evalq} and counting congruences to $0$ in the numerator and denominator, that $f^{m-11}_{CM}(\zeta^{m/d}) = 0$. 
\end{proof}

\section*{Acknowledgements}
The authors thank Julianna Tymoczko for introducing them to each other. We are also grateful to two anonymous referees for helpful suggestions. HM was partially supported by a Graduate Research Fellowship from the National Science Foundation. OP was partially supported by a Mathematical Sciences Postdoctoral Research Fellowship (\#1703696) from the National Science Foundation. 

This material is based upon work supported by the National Science Foundation Graduate Research Fellowship Program under Grant No. DGE-1752814. Any opinions, findings, and conclusions or recommendations expressed in this material are those of the authors and do not necessarily reflect the views of the National Science Foundation.

%
%

\bibliographystyle{amsalpha} 
\bibliography{exceptional}

\end{document}